\documentclass[10pt,twoside,reqno]{amsart}
\usepackage{amsmath}
\usepackage{amsfonts}
\usepackage{amssymb}
\usepackage{color}
\usepackage{mathrsfs}
\usepackage{cite}
\usepackage{times}
\newcommand{\R}{\mathbb R}

\newcommand{\Xq}{\mathbb{X}_T}

\newcommand{\ve}{\varepsilon} 
\newcommand{\rd}{\mathrm{d}}

\newcommand{\bean}{\begin{eqnarray}}
\newcommand{\eean}{\end{eqnarray}}
\DeclareMathAlphabet{\mathpzc}{OT1}{pzc}{m}{it}
\newtheorem{theorem}{Theorem}[section]

\newtheorem{proposition}[theorem]{Proposition}

\newtheorem{remark}[theorem]{Remark}
\numberwithin{equation}{section}

\title{A Remark on a Nonlocal-in-Time Heat Equation}

\begin{document}
%%%%%%%%%%%%%%%%

\author{Christoph Walker}
\address{Leibniz Universit\"at Hannover\\ Institut f\" ur Angewandte Mathematik \\ Welfengarten 1 \\ D--30167 Hannover\\ Germany}
\email{walker@ifam.uni-hannover.de}
%%%%%%%%%%%%%%%%
%
%\thanks{}
%
\date{\today}
%
%
%%%%%%%%%%%%%%%%
%%%%%%%%%%%%%%%%
\begin{abstract}
Schauder's fixed point theorem is used to derive the existence of solutions to a semilinear heat equation. The equation features a nonlinear term that depends on the time-integral of the unknown on the whole, a priori given, interval of existence.
\end{abstract}
%%%%%%%%%%%%%%%%
%%%%%%%%%%%%%%%%

%
%

\date{\today}

\maketitle

\section{Introduction}

This note is dedicated to the nonlocal problem
\begin{subequations}\label{1}
\begin{align}
\partial_t u-\mathrm{div}\big(d(x) \nabla u\big)+\varphi\left(\int_0^\infty a(s,x)u(s,x)\, \rd s\right) u&=f(t,x)\,,\quad (t,x)\in (0,\infty)\times\Omega\,,\\  u(0,\cdot)=u^0\,,\quad u\vert_{\partial\Omega}=0
\end{align}
\end{subequations}
on  a bounded smooth domain $\Omega\subset\R^n$ with $d\in C^1(\bar\Omega)$ such that $d(x)\ge d_0>0$ for $x\in \Omega$. The potential~$\varphi$,  the weight $a$, the initial datum $u^0$, and the right-hand side $f$ are  suitable given functions. %Moreover, also the existence time $T>0$ is a priori fixed.

Equation~\eqref{1} is used in the modeling of a biological nanosensor in the chaotic dynamics of a polymer chain in an aqueous solution and has been introduced and considered in \cite{StaraStara17,Stara18,Stara20,Stara22}. It can also be seen as a toy model for equilibrium states in age-structured diffusive populations (with $t$ referring to the age of individuals), see \cite{ChWCrelle11,ChWJEPE18} for instance.

Note that the unknown weighted time-integral 
$$
\bar{u}=\int_0^\infty a(s)\, u(s)\,\rd s
$$
 depends on the whole, a priori {\it given} interval of existence $(0,\infty)$ (the case of a bounded time interval $(0,T)$ is included in \eqref{1}, of course, by taking $a$ with bounded support). Hence, only global solutions are of interest. Moreover, \eqref{1} is no usual evolution problem satisfying a Volterra property since solutions at a time instant depend also  on future time instants. For the homogeneous version of~\eqref{1} on a bounded time interval $(0,T)$ with vanishing right-hand side $f$ and without weight~$a$, existence of weak solutions was derived in \cite{Stara18,Stara20} and strong solutions in \cite{ChWQAM}. The non-homogeneous problem~\eqref{1} on a bounded interval $(0,T)$ was investigated in \cite{Stara22} where, for unbounded potentials $\varphi$, a truncation approach and weak compactness methods were used to prove the existence of weak solutions under fairly general conditions. 

The purpose of this work is to propose an alternative approach to~\eqref{1} for deriving the existence of mild and strong solutions under slightly different conditions. This approach has been used in~\cite{ChWQAM} and may also be a template for other nonlocal problems. More precisely,
we shall use the fact that solutions to \eqref{1} may be written as  mild solutions in the form
\begin{equation}\label{u1}
u(t)=e^{t A(\bar{u})} u^0+\int_0^t e^{(t-s) A(\bar{u})} \, f(s)\,\rd s\,,\quad t\ge 0\,,
\end{equation}
where  $(e^{t A(\bar{u})})_{t\ge 0}$ is the contraction semigroup on $L_p(\Omega)$ generated by the operator $$A(\bar{u}):=\big[ w\mapsto \mathrm{div}\big(d(x) \nabla w\big)-\varphi(\bar{u})w\big]$$ subject to  Dirichlet boundary conditions (see below for details).
Integrating the representation~\eqref{u1}  yields the equivalent fixed point equation 
\begin{equation}\label{FP}
\bar{u}=\int_0^\infty a(t)\,e^{t A(\bar{u})} u^0\,\rd t+\int_0^\infty a(t)\,\int_0^t e^{(t-s) A(\bar{u})} \, f(s)\,\rd s\,\rd t
\end{equation}
for $\bar u$. We then shall focus on this fixed point equation and prove, in particular, that the right-hand side of~\eqref{FP} enjoys suitable compactness properties with respect to $\bar u$ that allow us to apply Schauder's theorem leading to the following existence result:

\begin{theorem}\label{T1}
Let  $a\in L_1(\R^+,L_\infty(\Omega))$, $\varphi\in C(\R,\R^+)$, and $f\in L_1(\R^+,L_\infty(\Omega))\cap L_{\infty,loc}(\R^+,L_p(\Omega))$ for some $p\in (\max\{1,n/2\},\infty)$ and let $u^0\in L_\infty(\Omega)$.
Then there is a mild solution $u\in C\big(\R^+,L_p(\Omega)\big)$
to~\eqref{1} such that
$$
\|u(t)\|_\infty\le \|u^0\|_\infty+\|f\|_{L_1((0,t),L_\infty(\Omega))}
\,,\quad t\in \R^+\,.
$$
If $f\in C^\theta(\R^+,L_q(\Omega))+C(\R^+,W_q^\theta(\Omega))$ with $\theta>0$ and $q\in (1,\infty)$, then $u$ is a strong solution with
 $$
u\in C\big(\R^+,L_q(\Omega)\big)\cap C^1\big(\dot{\R}^+,L_q(\Omega)\big)\cap C\big(\dot{\R}^+,W_{q}^2(\Omega)\big)\,,
$$
where $\dot{\R}^+:=(0,\infty)$.
%Moreover, if $u^0\ge 0$ and $f\ge 0$, then $u(t)\ge 0$ for~$t\in \R^+$.
\end{theorem}

In Section~\ref{Sec2} we  prove Theorem~\ref{T1}. The crucial compactness properties of the integral terms appearing on the right-hand side of~\eqref{FP} are postponed to Section~\ref{Sec4}. The proofs there are inspired by the works \cite{BHV77,B78} and may be extended to more general frameworks than the one considered herein for \eqref{1}, e.g. to other semilinear and possibly quasilinear equations (see Remark~\ref{R1} in this regard). Also some of the assumptions in Theorem~\ref{T1} may be weaken, e.g. for linearly bounded $\varphi$ or smoother $a$.

\section{Proof of Theorem~\ref{T1}}\label{Sec2}

%%%%%%%%%%%%%%%%%%%%%%%%%%%%%%%%%%%%%%%%%%%%%%%%%
%%%%%%%%%%%%%%%%%%%%%%%%%%%%%%%%%%%%%%%%%%%%%%%%%
\subsection*{Notation and Preliminaries} 
%%%%%%%%%%%%%%%%%%%%%%%%%%%%%%%%%%%%%%%%%%%%%%%%%
%%%%%%%%%%%%%%%%%%%%%%%%%%%%%%%%%%%%%%%%%%%%%%%%%

We use the notation
\begin{equation*}
W_{p,D}^{\alpha}(\Omega):=\left\{ 
\begin{array}{lcl} 
\{u\in  W_{p}^{\alpha}(\Omega)\,;\, u=0\ \text{on}\ \partial\Omega\} & \text{ if } & {\frac{1}{p}}<\alpha \le 2\, ,\\[3pt]
W_{p}^{\alpha}(\Omega) & \text{ if } & 0\le \alpha < {\frac{1}{p}}\,,
\end{array}
\right. 
\end{equation*}
and we write \mbox{$\mathcal{A}\in\mathcal{H}(W_{p,D}^2(\Omega),L_p(\Omega))$} if  $\mathcal{A}\in \mathcal{L}(W_{p,D}^2(\Omega),L_p(\Omega))$  generates an analytic semigroup $(e^{t\mathcal{A}})_{t\ge 0}$ on $L_p(\Omega)$. %, see \cite[I.Theorem 1.2.2]{LQPP}. 
 Recall that $$\big[ w\mapsto \mathrm{div}\big(d(x) \nabla w\big)\big]\in \mathcal{H}(W_{p,D}^2(\Omega),L_p(\Omega))$$ provided $d\in C^1(\bar\Omega))$ with $d(x)\ge d_0>0$ for $x\in \Omega$. \\

Since $\varphi$ is  uniformly continuous and bounded on bounded sets, it follows that (considered as Nemytskii operator)
\begin{equation}\label{300}
\varphi\in C\big(L_\infty(\Omega),L_\infty(\Omega)\big)\  \text{ is bounded on bounded sets}\,.
\end{equation}
Given $R_0>0$  denote by
$$
X:=\bar{\mathbb{B}}_{L_\infty(\Omega)}(0,R_0)
$$
the closed ball in $L_\infty(\Omega)$ of radius $R_0$ centered at the origin. Recall that $p\in (\max\{n/2 ,1\},\infty)$ and note that, given any $\bar{u}\in X$, the mapping $\varphi(\bar{u}):=[w\mapsto \varphi(\bar{u}) w]\in \mathcal{L}\big(L_p(\Omega)\big)$ satisfies
\begin{equation*}\label{op}
\| \varphi(\bar{u})\|_{\mathcal{L}(L_p(\Omega))}\le \|\varphi(\bar{u})\|_\infty\le \max_{[-R_0,R_0]}\varphi\,,\quad \bar{u}\in X\,.
\end{equation*}
Setting
$$
A(\bar{u})w:=\mathrm{div}\big(d \nabla w\big)-\varphi(\bar{u})w\,,\quad w \in W_{p,D}^2(\Omega)\,,
$$
it then follows from standard perturbation results that %the perturbation result \cite[I.Theorem 1.3.1]{LQPP} entails that
\begin{equation*}\label{301y}
A(\bar{u})\in\mathcal{H}(W_{p,D}^2(\Omega),L_p(\Omega))\,,\quad \bar{u}\in X \,.
\end{equation*}
In fact, since $\varphi$ is nonnegative,  $(e^{t A(\bar{u})})_{t\ge 0}$ is a positive contraction semigroup on each $L_q(\Omega)$ for $q\in (1,\infty]$ (which, however, is {\it not} strongly continuous for $q=\infty$), hence
\begin{equation}\label{cs}
\| e^{t A(\bar{u})}\|_{\mathcal{L}(L_q(\Omega))}\le 1\,,\quad t\ge 0 \,,\quad q\in (1,\infty]\,.
\end{equation}
Moreover, we have $s(A(\bar{u}))\le s_0<0$ for its spectral bound with $s_0$ denoting the spectral bound of the operator $\big[ w\mapsto \mathrm{div}\big(d \nabla w\big)\big]$. 
It then follows from \cite[II.Lemma~5.1.3]{LQPP} that there is $\nu>0$
and, given $2\theta\in [0,2]\setminus\{1/p\}$, there is $M(R_0)\ge 1$  such that
\begin{equation}\label{smoot}
\|e^{t A(\bar{u})}\|_{\mathcal{L}(L_p(\Omega),W_{p,D}^{2\theta}(\Omega))}\le M(R_0)\,e^{-\nu t}\, t^{-\theta}\,,\quad t>0\,,\quad \bar{u}\in X\,.
\end{equation}
In the following we fix $2\theta\in (n/p,2)$ and note the compact embedding
\begin{equation}\label{im}
W_{p,D}^{2\theta}(\Omega)\stackrel{c}{\hookrightarrow} L_\infty(\Omega)\hookrightarrow L_p(\Omega)\,.
\end{equation}
Let us also observe that, given $t>0$ and $\bar{u}, \bar{v}\in X$, we have
\begin{align}
e^{t A(\bar{u})}-e^{t A(\bar{v})} &=-\int_0^t \frac{\rd }{\rd s} e^{(t-s) A(\bar{u})} e^{s A(\bar{v})}\,\rd s = -\int_0^t e^{(t-s) A(\bar{u})}\,\big(\varphi(\bar{u})-\varphi(\bar{v})\big)\, e^{s A(\bar{v})}\,\rd s\,.\label{form}
\end{align}
We then use  \eqref{smoot}  and \eqref{im} to get
\begin{align}
\|e^{t A(\bar{u})}-e^{t A(\bar{v})}&\|_{\mathcal{L}(L_p(\Omega),L_\infty(\Omega))}\le c\, \|e^{t A(\bar{u})}-e^{t A(\bar{v})}\|_{\mathcal{L}(L_p(\Omega),W_{p,D}^{2\theta}(\Omega))}\notag\\
 &\le \int_0^t \|e^{(t-s) A(\bar{u})}\|_{\mathcal{L}(L_p(\Omega),W_{p,D}^{2\theta}(\Omega))}\, \|\varphi(\bar{u})-\varphi(\bar{v})\|_{\mathcal{L}(L_p(\Omega))}\, \| e^{s A(\bar{v})}\|_{\mathcal{L}(L_p(\Omega))}\,\rd s\notag\\
&\le c(R_0)\, e^{-\nu t}\, t^{1-\theta} \, \|\varphi(\bar{u})-\varphi(\bar{v})\|_{\infty}\,.\label{lll}
\end{align}
We are now in a position to provide the proof of Theorem~\ref{T1}.

%%%%%%%%%%%%%%%%%%%%%%%%%%%%%%%%%%%%%%%%%
%%%%%%%%%%%%%%%%%%%%%%%%%%%%%%%%%%%%

\subsection*{Proof of Theorem~\ref{T1}}\label{Sec3b}

%%%%%%%%%%%%%%%%%%%%%%%%%%%%%%%%%
%%%%%%%%%%%%%%%%%%%%%%%%%%%%%%%%%%%%

 We show that the mapping
\begin{equation*}\label{Phi}
\Phi(\bar{u}):=\int_0^\infty a(t)\, e^{t A(\bar{u})} u^0\,\rd t+\int_0^\infty a(t)\,\int_0^t e^{(t-s) A(\bar{u})} \, f(s)\,\rd s\,\rd t\,,\quad \bar{u}\in X=\bar{\mathbb{B}}_{L_\infty(\Omega)}(0,R_0)\,,
\end{equation*}
with
$$
R_0:=\|a\|_{L_1(\R^+,L_\infty(\Omega))}\big(\|u^0\|_\infty+\|f\|_{L_1(\R^+,L_\infty(\Omega))}\big)
$$
has a fixed point. Note first from \eqref{cs}  that $\|\Phi(\bar{u})\|_\infty\le R_0$ for  $\bar{u}\in X$,
hence $\Phi:X\rightarrow X$. Next,  due to~\eqref{lll} and~\eqref{im} we obtain, for $\bar u, \bar v\in X$,
\begin{align*}
\|\Phi(\bar{u})-\Phi(\bar{v})\|_\infty
&\le
 \int_0^T \|a(t)\|_\infty\,\|e^{t A(\bar{u})}-e^{t A(\bar{v})}\|_{\mathcal{L}(L_p(\Omega),L_\infty(\Omega))}\, \|u^0\|_p\,\rd t \\
&\quad + \int_0^T \|a(t)\|_\infty\,\int_0^t \|e^{(t-s) A(\bar{u})}-e^{(t-s) A(\bar{v})}\|_{\mathcal{L}(L_p(\Omega),L_\infty(\Omega))}\, \|f(s)\|_p\,\rd s\,\rd t \notag\\
&\le c(R_0) \int_0^\infty \|a(t)\|_\infty\, e^{-\nu t}\, t^{1-\theta}\,\rd t\, \|u^0\|_p\, \|\varphi(\bar{u})-\varphi(\bar{v})\|_{\infty} \notag\\
&\quad + c(R_0) \int_0^\infty \|a(t)\|_\infty\, \int_0^t e^{-\nu (t-s)}\, (t-s)^{1-\theta} \, \|f(s)\|_p\,\rd s\, \rd t\, \|\varphi(\bar{u})-\varphi(\bar{v})\|_{\infty} \notag\\
\end{align*}
and hence, recalling $\nu>0$,
\begin{align}
\|\Phi(\bar{u})-\Phi(\bar{v})\|_\infty
\le c(R_0)\, \|a\|_{L_1(\R^+,L_\infty(\Omega))}\,\big(\|u^0\|_p+\|f||_{L_1(\R^+,L_p(\Omega))}\big)\, \|\varphi(\bar{u})-\varphi(\bar{v})\|_{\infty}\,.\label{PB}
\end{align}
Therefore, $\Phi\in C(X, X)$ according to \eqref{300}. From Proposition~\ref{C1} and Proposition~\ref{P0} we deduce that \mbox{$\Phi\in C(X,X)$} has precompact image so that Schauder's fixed point theorem yields $\bar{u}\in X$ such that $\bar{u}=\Phi(\bar{u})$. We may then  define $u$ by~\eqref{u1}
% \begin{equation*}
%u(t):=e^{t A(\bar{u})} u^0+\int_0^t e^{(t-s) A(\bar{u})} \, f(s)\,\rd s\,,\quad t\in \R^+\,,
%\end{equation*}
in order to obtain a mild solution to \eqref{1} which belongs to $C(\R^+,L_p(\Omega))$ due to \cite[II.Theorem~5.3.1]{LQPP}  since $f\in L_{\infty,loc}(\R^+,L_p(\Omega))$.

Finally, if $f\in C^\theta(\R^+,L_q(\Omega))+C(\R^+,W_q^\theta(\Omega))$ for some $\theta>0$ and $q\in (1,\infty)$, then $u$ is a strong solution with
 the regularity properties stated in Theorem~\ref{T1}, see \cite[II.Theorem~1.2.1, II.Theorem~1.2.2]{LQPP}. This proves Theorem~\ref{T1}.\qed

\begin{remark}
If $\varphi$ is locally Lipschitz continuous, then one may derive the existence and uniqueness of a solution using Banach's fixed point argument for $\Phi:X\to X$ provided that
$$R_0=\|a\|_{L_1(\R^+,L_\infty(\Omega))}\big(\|u^0\|_\infty+\|f\|_{L_1(\R^+,L_\infty(\Omega))}\big) $$
is small enough.
\end{remark}

%%%%%%%%%%%%%%%%%%%%%%%%%%%%%%%%%%%%%%%%%%%%
%%%%%%%%%%%%%%%%%%%%%%%%%%%%%%%%%%%%%%%%%%%%
%%%%%%%%%%%%%%%%%%%%%%%%%%%%%%%%%%%%%%%%%%%%
\section{Compactness Properties}\label{Sec4}
%%%%%%%%%%%%%%%%%%%%%%%%%%%%%%%%%%%%%%%%%%%%
%%%%%%%%%%%%%%%%%%%%%%%%%%%%%%%%%%%%%%%%%%%%
%%%%%%%%%%%%%%%%%%%%%%%%%%%%%%%%%%%%%%%%%%%%

We provide the compactness results used in the proof of Theorem~\ref{T1}. This section relies on the papers \cite{BHV77,B78} and adapts these ideas to our setting.

We first consider the non-homogeneous part. 

\begin{proposition}\label{C1}
Let $a\in L_1(\R^+,L_\infty(\Omega))$  
and  $f\in L_1(\R^+,L_\infty(\Omega))\cap L_{\infty,loc}(\R^+,L_p(\Omega))$ for some  $p\in (\max\{n/2, 1\},\infty)$. Define
$$
\mathcal{F}(\bar u):=\int_0^\infty a(t)\int_0^t e^{(t-s)A(\bar u)}\, f(s)\,\rd s\,\rd t\,,\quad \bar u\in X\,.
$$
Then the set
$
\left\{ \mathcal{F}(\bar u)\,;\, \bar u\in X \right\} 
$
 is precompact in $L_\infty(\Omega)$.
\end{proposition}

\begin{proof} Given $T>0$ introduce
$$
\Xq:=C([0,T],L_\infty(\Omega))
$$
and
$$
F(\bar u)(t):=\int_0^t e^{(t-s)A(\bar u)}\, f(s)\,\rd s\,,\qquad t\in [0,T]\,,\quad \bar u\in X=\bar{\mathbb{B}}_{L_\infty(\Omega)}(0,R_0)\,.
$$
It suffices to prove that $F\in C\left(X,\Xq\right)$ is compact for every $T>0$ since the assertion then follows by a diagonal sequence argument and the assumption $a\in L_1(\R^+,L_\infty(\Omega))$.

{\bf (i)} Since  $f\in L_{\infty,loc}(\R^+,L_p(\Omega))$, we infer from \cite[II.Theorem 5.3.1]{LQPP} that $F(\bar u)\in C([0,T],W_p^{2\theta}(\Omega))$ and hence 
$F(\bar u)\in \Xq$ for $\bar u\in X$ by~\eqref{im}.
Moreover, given $\bar u, \bar v\in X$ and $t\in [0,T]$, we have, as in~\eqref{PB},
\begin{align*}
\|F(\bar u)(t)-F(\bar v)(t)\|_\infty %&\le \int_0^t \| e^{(t-s) A(\bar{u})}-e^{(t-s) A(\bar{v})}\|_{\mathcal{L}(L_p(\Omega),L_\infty(\Omega))}\,\|f(s) \|_p \,\rd s\\
%& \le c(R_0) \int_0^t e^{-\nu (t-s)}\, (t-s)^{1-\theta} \, \| f(s)\|_p \,\rd s\, %\|\varphi(\bar{u})-\varphi(\bar{v})\|_{\infty}
\le c(R_0)\, \| f\|_{L_1(\R^+,L_p(\Omega))}\,\|\varphi(\bar{u})-\varphi(\bar{v})\|_{\infty}\,.
\end{align*}
Therefore, $F\in C\left(X,\Xq\right)$ owing to~\eqref{300}.

{\bf (ii)} In order to prove that $F\in C\left(X,\Xq\right)$ has precompact image we use an idea inspired by \cite{BHV77,B78}: for fixed $\lambda>0$ we first show that 
\begin{align}\label{AA}
\left\{e^{\lambda A(\bar u)}F(\bar u)\,;\, \bar u\in X\right \}\ \text{ is precompact in }\ \Xq\,.
\end{align}
To this end,  we note from~\eqref{smoot}  and~\eqref{cs} that, for $\bar u\in X$ and $t\in [0,T]$,
\begin{align}\label{Fp}
\|e^{\lambda A(\bar u)}F(\bar u)(t)\|_{W_{p,D}^{2\theta}(\Omega)}\le c(R_0)\,\lambda^{-\theta}\,\|F(\bar u)(t)\|_p\le c(R_0)\,\lambda^{-\theta} \|f\|_{L_1(\R^+,L_p(\Omega))}\,.
\end{align}
That is, invoking~\eqref{im}, the set
\begin{align}\label{comp}
\left\{e^{\lambda A(\bar u)}F(\bar u)(t)\,;\, \bar u\in X\right \}\ \text{ is precompact in }\ L_\infty(\Omega)
\end{align}
for every $t\in [0,T]$. Before proceeding let us note that, given $\bar u\in X$, $\delta>0$ and $h\ge 0$, we have
\begin{align*}
e^{(\delta +h) A(\bar{u})}-e^{\delta A(\bar{u})} &=\int_0^h \frac{\rd }{\rd s} e^{(\delta+s) A(\bar{u})} \,\rd s= \int_0^h e^{(\frac{\delta}{2}+s) A(\bar{u})} A(\bar{u}) e^{\frac{\delta}{2} A(\bar{u})}\,\rd s 
\end{align*}
so that, using~\eqref{smoot},
\begin{align*}
\|&e^{(\delta +h) A(\bar{u})}-e^{\delta A(\bar{u})}\|_{\mathcal{L}(L_p(\Omega),W_{p,D}^{2\theta}(\Omega))}\notag\\ &\quad\le  \int_0^h \|e^{(\frac{\delta}{2}+s) A(\bar{u})}\|_{\mathcal{L}(L_p(\Omega),W_{p,D}^{2\theta}(\Omega))} \, \|A(\bar{u})\|_{\mathcal{L}(W_{p,D}^{2}(\Omega),L_p(\Omega))}\, \|e^{\frac{\delta}{2} A(\bar{u})}\|_{\mathcal{L}(L_p(\Omega),W_{p,D}^{2}(\Omega))} \,\rd s\notag\\
&\quad\le c(R_0) \,e^{-\nu\delta}\, \left(\frac{\delta}{2}\right)^{-1}\int_0^h \left(\frac{\delta}{2}+s\right)^{-\theta}\,\rd s\,.
\end{align*}
Consequently, invoking~\eqref{im},
\begin{align}\label{est}
\|e^{(\delta +h) A(\bar{u})}-e^{\delta A(\bar{u})}\|_{\mathcal{L}(L_p(\Omega),L_\infty(\Omega))} 
\le c(R_0) \, e^{-\nu\delta}\,\delta^{-1-\theta} h\,,\qquad h\ge 0\,,\quad \delta>0\,,\quad \bar u\in X\,.
\end{align}
Next, given $\bar u\in X$ and $0\le t\le t+h\le T$ we have, using~\eqref{im},
\begin{align*}
\|e^{\lambda A(\bar u)}F(\bar u)&(t+h)-e^{\lambda A(\bar u)}F(\bar u)(t)\|_\infty\\
 &\le c\int_t^{t+h} \| e^{(\lambda+t+h-s) A(\bar{u})}\|_{\mathcal{L}(L_p(\Omega),L_\infty(\Omega))}\,\|f(s) \|_p \,\rd s\\
&\quad  + c\int_0^{t} \| e^{(\lambda+t+h-s) A(\bar{u})}-e^{(\lambda+t-s) A(\bar{u})}\|_{\mathcal{L}(L_p(\Omega),L_\infty(\Omega))}\,\|f(s) \|_p \,\rd s\,.
\end{align*}
We use \eqref{smoot}-\eqref{im}  once more along with \eqref{est} to get
\begin{align*}
\|e^{\lambda A(\bar u)}&F(\bar u)(t+h)-e^{\lambda A(\bar u)}F(\bar u)(t)\|_\infty\\
 &\le c(R_0)\int_t^{t+h} (\lambda+t+h-s)^{-\theta} \,\|f(s) \|_p \,\rd s  + c(R_0) \, h \int_0^{t}  (\lambda+t-s)^{-1-\theta}\,\|f(s) \|_p \,\rd s\\
&\le c(R_0)\, \lambda^{-\theta}\int_t^{t+h} \|f(s) \|_p \,\rd s + c(R_0)\,h\,\lambda^{-1-\theta}  \,\|f \|_{L_1(\R^+,L_p(\Omega))} \,.
\end{align*}
Therefore, since $f\in L_1(\R^+,L_p(\Omega))$ we deduce that
\begin{align}\label{conti}
\lim_{h\to 0}\,\sup_{\bar u\in X}\,\|e^{\lambda A(\bar u)}&F(\bar u)(\cdot+h)-e^{\lambda A(\bar u)}F(\bar u)\|_{\mathbb{X}_T}=0\,,\quad \lambda>0\,.
\end{align}
Gathering~\eqref{conti} and~\eqref{comp} we conclude that $\left\{e^{\lambda A(\bar u)}F(\bar u)\,;\, \bar u\in X\right \}$ is indeed precompact in $\mathbb{X}_T$ due to the Arzel\`{a}-Ascoli Theorem.  

{\bf (iii)} Next, we claim that
\begin{align}\label{stst}
\lim_{\lambda\to 0}\,\sup_{\bar u\in X}\,\|e^{\lambda A(\bar u)}&F(\bar u)-F(\bar u)\|_{\Xq}=0\,.
\end{align}
Let $\delta\in (0,T)$. Using~\eqref{cs} we have, for $0\le t\le \delta$,
\begin{align}%\label{ee1}
\|e^{\lambda A(\bar u)}F(\bar u)(t)-F(\bar u)(t)\|_\infty&\le 2 \|F(\bar u)(t)\|_\infty%\le 2\int_0^t\|e^{(t-s) A(\bar u)}\|_{\mathcal{L}(L_\infty(\Omega))}\, \| f(s)\|_\infty\,\rd s\notag\\
\le  2\int_0^\delta \| f(s)\|_\infty\,\rd s\,.\label{ee1}
\end{align}
On the other hand, for $\delta\le t\le T$, we use~\eqref{cs} to get
\begin{align*}
\|e^{\lambda A(\bar u)}F(\bar u)(t)-F(\bar u)(t)\|_\infty &\le \|e^{\lambda A(\bar u)}F(\bar u)(t)-e^{(\lambda+\delta) A(\bar u)}F(\bar u)(t-\delta)\|_\infty\\
&\quad+\|e^{(\lambda+\delta) A(\bar u)}F(\bar u)(t-\delta)-e^{\delta A(\bar u)}F(\bar u)(t-\delta)\|_\infty\\
&\quad+\|e^{\delta A(\bar u)}F(\bar u)(t-\delta)-F(\bar u)(t)\|_\infty\\
&\le 2\, \|e^{\delta A(\bar u)}F(\bar u)(t-\delta)-F(\bar u)(t)\|_\infty\\
&\quad
+\|e^{(\lambda+\delta) A(\bar u)}F(\bar u)(t-\delta)-e^{\delta A(\bar u)}F(\bar u)(t-\delta)\|_\infty\,.%\\
%&=2\,\left\| \int_{t-\delta}^t e^{(t-s) A(\bar u)}\, f(s)\,\rd s\right\|_\infty\\
%&\quad +\|e^{(\lambda+\delta) A(\bar u)}F(\bar u)(t-\delta)-e^{\delta A(\bar u)}F(\bar u)(t-\delta)\|_\infty\,.
\end{align*}
For the first term on the right-hand side we use~\eqref{cs}  to estimate
\begin{align*}
\|e^{\delta A(\bar u)}F(\bar u)(t-\delta)-F(\bar u)(t)\|_\infty=\left\| \int_{t-\delta}^t e^{(t-s) A(\bar u)}\, f(s)\,\rd s\right\|_\infty&\le \int_{t-\delta}^t \|f(s)\|_\infty\,\rd s
\end{align*}
while we use~\eqref{est} and~\eqref{Fp} for the second term to obtain
\begin{align*}
\|e^{(\lambda+\delta) A(\bar u)}F(\bar u)(t-\delta)-e^{\delta A(\bar u)}F(\bar u)(t-\delta)\|_\infty&\le
c(R_0)\, \delta^{-1-\theta}\, \lambda\, \|F(\bar u)(t-\delta)\|_p\\
&\le c(R_0)\, \delta^{-1-\theta}\, \lambda\,\|f\|_{L_1(\R^+,L_p(\Omega))}\,.
\end{align*}
Gathering these estimates we derive, for $\delta\le t\le T$,
\begin{equation}
\begin{split}\label{ee2}
\|e^{\lambda A(\bar u)}F(\bar u)(t)-F(\bar u)(t)\|_\infty &\le 
 2\int_{t-\delta}^t \|f(s)\|_\infty\,\rd s +c(R_0)\, \delta^{-1-\theta}\, \lambda\,\|f\|_{L_1(\R^+,L_p(\Omega))}\,.
\end{split}
\end{equation}
Since $f\in L_1(\R^+,L_\infty(\Omega))$ we may first choose $\delta>0$ small enough and then let $\lambda$ tend to zero to conclude from~\eqref{ee1} and~\eqref{ee2} that~\eqref{stst} indeed holds true.

{\bf (iv)} %It follows from~\eqref{stst} and \eqref{AA} that $\left\{F(\bar u)\,;\, \bar u\in X\right \}$ is totally bounded in $\Xq$, hence the assertion.
Let $\ve>0$ be arbitrary. Then, due to~\eqref{stst}, there is $\lambda_0>0$ such that
\begin{align}\label{A1}
\|e^{\lambda_0 A(\bar u)}&F(\bar u)-F(\bar u)\|_{\Xq}\le\frac{\ve}{3}\,,\quad \bar u\in X\,,
\end{align}
while \eqref{AA} yields finitely many $\bar u_1,\ldots,\bar u_N\in X$ such that for every $\bar u\in X$ there exists $k\in{1,\ldots,N}$ such that  
\begin{align}\label{A2}
\|e^{\lambda_0 A(\bar u)} F(\bar u)-e^{\lambda_0 A(\bar u_k)} F(\bar u_k)\|_{\Xq}\le \frac{\ve}{3}\,.
\end{align}
Hence $\| F(\bar u)- F(\bar u_k)\|_{\Xq}\le\ve$
so that $\left\{F(\bar u)\,;\, \bar u\in X\right \}$ is totally bounded in $\Xq$. This proves the assertion.
\end{proof}

%%%%%%%%%%%%%%%%%%%%%%%%%%%%%%%%%%%%%%%%%%%%%%%%%%%%%%%%%%%%%%%%%%%%%%%%%%%%%%%%%%%%%%%%%%%%%%%%%%%%%%
%%%%%%%%%%%%%%%%%%%%%%%%%%%%%%%%%%%%%%%%%%%%%%%%%%%%%%%%%%%%%%%%%%%%%%%%%%%%%%%%%%%%%%%%%%%%%%%%%%%%%%
%%%%%%%%%%%%%%%%%%%%%%%%%%%%%%%%%%%%%%%%%%%%%%%%%%%%%%%%%%%%%%%%%%%%%%%%%%%%%%%%%%%%%%%%%%%%%%%%%%%%%%

We prove a compactness result for the part involving the initial condition:

\begin{proposition}\label{P0}
Let $a\in L_1(\R^+,L_\infty(\Omega))$.
Given $u^0\in L_\infty(\Omega)$  define
$$
\mathcal{G}(\bar u):=\int_0^\infty a(t)\,e^{tA(\bar u)}\, u^0\,\rd t\,,\quad \bar u\in X=\bar{\mathbb{B}}_{L_\infty(\Omega)}(0,R_0)\,.
$$
Then the set
$
\left\{ \mathcal{G}(\bar u)\,;\, \bar u\in X \right\} 
$
 is precompact in $L_\infty(\Omega)$.
\end{proposition}

\begin{proof}
Let $\lambda>0$ and set
$$
\mathcal{G}_\lambda(\bar u):=\int_0^\infty a(t)\,e^{(\lambda+t)A(\bar u)}\, u^0\,\rd t\,,\quad \bar u\in X\,.
$$
Similarly as in Proposition~\ref{C1} we infer that 
\begin{align}\label{AAa}
\left\{\mathcal{G}_\lambda(\bar u)\,;\, \bar u\in X\right \}\ \text{ is precompact in }\ L_\infty(\Omega)\,.
\end{align}
Taking $\delta>0$ and using~\eqref{cs}, \eqref{im}, and \eqref{est} we then get
\begin{align*}
\|\mathcal{G}_\lambda(\bar u)-\mathcal{G}(\bar u)\|_\infty
&\le 2\int_0^\delta \| a(t)\|_\infty \,\rd t \,\|u^0\|_\infty+\int_\delta^\infty \| a(t)\|_\infty \, \big\|e^{(\lambda+t) A(\bar u)}u^0-e^{t A(\bar u)}u^0\|_\infty\, \rd t\\
&\le  2\int_0^\delta \| a(t)\|_\infty \,\rd t \,\|u^0\|_\infty+c(R_0)\,\delta^{-1-\theta}\,\|a\|_{L_1(\R^+,L_\infty(\Omega))}\,\|u^0\|_p\,\lambda\,.
\end{align*}
Since  $a\in L_1(\R^+,L_\infty(\Omega))$ we may choose $\delta>0$ small to make the first term small and then let $\lambda$ go to zero  to conclude
\begin{align}\label{ststa}
\lim_{\lambda\to 0}\,\sup_{\bar u\in X}\,\|\mathcal{G}_\lambda(\bar u)-\mathcal{G}(\bar u)\|_\infty=0\,.
\end{align}
Combining \eqref{AAa} and \eqref{ststa} we deduce that $\left\{ \mathcal{G}(\bar u)\,;\, \bar u\in X \right\}$ is precompact in $L_\infty(\Omega)$.
\end{proof}
%%%%%%%%%%%%%%%%%%%%%%%%%%%%%%%%%%%%%%%%%%%%%%%%%%%%%%%%%%%%%%%%%%%%%%%%%%%%%%%%%%%%%%%%%%%%%%%%%%%%%%
%%%%%%%%%%%%%%%%%%%%%%%%%%%%%%%%%%%%%%%%%%%%%%%%%%%%%%%%%%%%%%%%%%%%%%%%%%%%%%%%%%%%%%%%%%%%%%%%%%%%%%
%%%%%%%%%%%%%%%%%%%%%%%%%%%%%%%%%%%%%%%%%%%%%%%%%%%%%%%%%%%%%%%%%%%%%%%%%%%%%%%%%%%%%%%%%%%%%%%%%%%%%%
 
\begin{remark}\label{R1}
The compactness results of this section rely on properties~\eqref{cs}, \eqref{smoot}, and \eqref{im} and thus may also be derived for truly quasilinear operators $A(\bar u)$  based on the stability estimates of \cite[II.Section~5]{LQPP}. One has, however, to replace the set $X\subset L_\infty(\Omega)$ by a subset of $W_{p,D}^{2\theta}(\Omega)$ with $\theta\in (0,2)\setminus\{1/p\}
$ and to impose suitable assumptions on the data (e.g. the weight $a$ has to be sufficiently smooth in $x$).
\end{remark}

\end{document}